\documentclass[11pt]{article}

\usepackage[margin=1in]{geometry}
\usepackage{amsmath,amssymb,amsthm,mathtools}
\usepackage{graphicx}
\usepackage{hyperref}
\hypersetup{hidelinks}
\usepackage{bm}
\usepackage{booktabs}
\usepackage[font=small]{caption}
\usepackage{tcolorbox}

\newtheorem{theorem}{Theorem}

\newtheorem{corollary}{Corollary}
\theoremstyle{definition}

\theoremstyle{remark}

\newcommand{\R}{\mathbb{R}}
\newcommand{\V}{\mathbb{V}}
\newcommand{\beq}{\begin{equation}}
\newcommand{\eeq}{\end{equation}}
\newcommand{\ba}{\begin{array}}
\newcommand{\ea}{\end{array}}
\newcommand{\la}{\langle}
\newcommand{\ra}{\rangle}
\DeclareMathOperator{\argmin}{argmin}
\DeclareMathOperator{\dom}{dom}

\newcommand{\mat}[1]{\bm{#1}}

\providecommand{\0}{\boldsymbol{0}}

\providecommand{\hh}{\boldsymbol{h}}

\renewcommand{\ss}{\boldsymbol{s}}

\providecommand{\vv}{\boldsymbol{v}}

\providecommand{\xx}{\boldsymbol{x}}
\providecommand{\yy}{\boldsymbol{y}}

\title{
\vspace*{-1.98cm}
	\textbf{
		Primal Acceleration of Newton's Method}}
\author{Nikita Doikov~\thanks{School of Operations Research and Information Engineering (ORIE), Cornell University, Ithaca, NY, USA. \texttt{nikita.doikov@cornell.edu}}}
\date{August 21, 2026}

\begin{document}
	
	\maketitle	
	\begin{abstract}
		We develop a new direct accelerated Newton method
		for minimizing convex functions with Lipschitz continuous Hessian.
		The algorithm uses only primal variables and performs just one linear solve per iteration. With a simple predetermined choice of parameters, it achieves the global convergence rate of $O(1/k^3)$ in terms of the functional residual. To the best of our knowledge,
		this is the first second-order method for this problem class attaining this rate 
		while relying solely on one linear system solve per iteration (without solving auxiliary nonlinear regularized subproblems,
		such as cubic regularization,
		performing nonlinear parameter searches, or using dual extragradient corrections). Our method can be implemented in a Hessian-free way, using an inexact linear system solver,
		while preserving the fast global rate. We further extend our construction to arbitrary geometry through Bregman divergence, and to composite optimization problems.
		
	\end{abstract}

	\section{Introduction}

	Let us consider the standard unconstrained minimization problem:
	\beq \label{MainProblem}
	\ba{rcl}
	\min\limits_{\xx \in \R^n} f(\xx),
	\ea
	\eeq
	where the objective $f: \R^n \to \R$ is convex and several times differentiable.
	
	The Fast Gradient Method (FGM) of Nesterov~\cite{nesterov1983method} 
	is the optimal first-order algorithm for solving problem~\eqref{MainProblem}, which can be implemented in the following momentum form,
	starting from an arbitrary initialization $\xx_0 = \yy_0 \in \R^n$, we iterate:
	\beq \label{FGM}
	\left. \textbf{FGM:} \qquad
	\ba{rcl} 
	\xx_{k + 1} & = & \yy_k -  \alpha_k \nabla f(\yy_k) \\[10pt]
	\yy_{k + 1} & = & \xx_{k + 1} + \beta_k (\xx_{k + 1} - \xx_k)
	\ea
	\right\},  \quad k \geq 0.
	\eeq
	Parameter $\alpha_k > 0$ plays the role of the step-size, and $\beta_k \in [0, 1]$ is the momentum coefficient,
	which specifies how fast we forget the history. The choice $\beta_k = 0$ (no momentum) corresponds to the classic gradient descent, while setting, for example $\beta_k := \frac{k}{k + 3}$, yields the accelerated rate of convergence of $O(1 / k^2)$ (see Table~1 below for the detailed comparison of rates).
	
	In this paper, we study a direct generalization of the classic FGM iterates~\eqref{FGM} to second-order optimization methods,
	incorporating the Hessian (or its approximation) $\mat{H}_k \approx \nabla^2 f(\yy_k)$ in the following most natural way,
	which we call the \textit{primal acceleration} of Newton method:
	\beq \label{FNM}
	\left. \textbf{This paper:}\qquad 
	\ba{rcl} 
	\xx_{k + 1} & = & \yy_k - \bigl(  \mat{H}_k + \frac{1}{\alpha_k} \mat{I} \bigr)^{-1} \nabla f(\yy_k) \\[10pt]
	\yy_{k + 1} & = & \xx_{k + 1} + \beta_k (\xx_{k + 1} - \xx_k)
	\ea \right\},  \quad k \geq 0.
	\eeq
	Note that method~\eqref{FNM} operates solely in the primal space of variables,
	without aggregating the dual objects (gradients), in contrast to the estimate sequence framework~\cite{nesterov2018lectures}, and without performing any extragradient corrections.
	Moreover, setting $\mat{H}_k := \0$ immediately recovers the classic FGM.
	Most importantly, algorithm~\eqref{FNM} is very simple to implement as it requires only \textit{one linear system solve per iteration}. At the same time, using the exact Hessian $\mat{H}_k := \nabla^2 f(\yy_k)$
	and a predefined schedule of second-order step-sizes $\alpha_k$ equips the algorithm
	with accelerated $O(1 / k^3)$ rate on convex functions with Lipschitz Hessian. To the best of our knowledge, this is the first
	second-order method with this rate on our problem class with such simple  per-iteration work, as 
	all previously known acceleration techniques required solving auxiliary regularized subproblems (such as cubic regularization) or performing nonlinear parameter searches. We summarize convergence rates for our new approach and for previously known methods in the following table.
	
	\begin{table}[htbp]
		\centering
		\renewcommand{\arraystretch}{1.3} 
		\begin{tabular}{l l l c}
			\multicolumn{4}{c}{\textbf{First-order methods: $\mat{H}_k = \0$ in~\eqref{FNM}}} \\
				\midrule
			{Method} & \multicolumn{2}{c}{Parameters} & {Convergence Rate, $F_k$} \\
			\midrule
			Gradient Descent     & $\alpha_k = \frac{1}{L_1}$ & $\beta_k = 0$ & $O\bigl( \frac{L_1 R^2}{k}  \bigr)$ \\
			Fast Gradient Method~\cite{nesterov1983method} & $\alpha_k = \frac{1}{L_1} \cdot \frac{k + 1}{k + 2}$ & $\beta_k = \frac{k}{k + 3}$ & $O\bigl( \frac{L_1 R^2}{k^2}  \bigr)$ \\
			\midrule
			\multicolumn{4}{c}{\textbf{Second-order methods: $\mat{H}_k = \nabla^2 f(\yy_k)$  in~\eqref{FNM}}} \\
			\midrule
			{Method} & \multicolumn{2}{c}{Parameters} & {Convergence Rate, $F_k$} \\
			\midrule
			Classical Newton's Method            & $\alpha_k = +\infty$ & $\beta_k = 0$  & No global convergence \\
			Cubic Regularization~\cite{nesterov2006cubic}            & $\alpha_k = \frac{2}{L_2 \| \xx_{k + 1} - \xx_k \|_2}$ & $\beta_k = 0$  & $O\bigl( \frac{L_2 D^3}{k^2} \bigr)$ \\
			Gradient Regularization~\cite{mishchenko2023regularized}        & $\alpha_k = \frac{2}{ \sqrt{L_2 \| \nabla f(\xx_k) \|_2} }$ & $\beta_k = 0$  & $O\bigl( \frac{L_2 D^3}{k^2} + e^{-k} \bigr)$ \\[10pt]
			\textbf{This paper:}             & $\alpha_k = \frac{1}{ 3 L_2 R} \cdot \frac{(k + 1)(k + 2)}{k + 3}$ & $\beta_k = \frac{k}{k + 4}$  & $O\bigl( \frac{L_2 R^3}{k^3} \bigr)$ \\[8pt]
			Optimal~\cite{monteiro2013accelerated,kovalev2022first,carmon2022optimal} & 
			\multicolumn{2}{c}{Nonlinear subproblem for each $k$} &
			$O\bigl( \frac{L_2 R^3}{k^{7/2}} \bigr)$  \\
			\bottomrule
		\end{tabular}
		\caption{Summary of optimization methods and their convergence rates in terms of the functional residual $F_k := f(\xx_k) - f^{\star}$, and $O(\cdot)$ hides an absolute numerical constant. We denote by $L_1$ the Lipschitz constant of the gradient, and by $L_2$ the Lipschitz constant of the Hessian, both measured w.r.t. the standard Euclidean norm. We denote by $D$ the diameter of the initial sublevel set, and $R \geq \| \xx_0 - \xx^{\star} \|_2$ is an arbitrary bound for the initial distance to a solution $\xx^\star$, which we assume to exist. }
		\label{tab:algorithms}
	\end{table}

	\paragraph{Related Literature.}

	Newton's method is a powerful and well-developed approach for solving nonlinear optimization
	problems across various domains. See~\cite{polyak2007newton, nocedal2006numerical, nesterov2018lectures}
	for the detailed bibliography and for the classic results.
	One of the fundamental research questions related to the Newton-type methods
	is the question of its \textit{global convergence}, starting from an arbitrary initialization.
	
	A significant progress in this direction was achieved after the development
	of the \textit{cubic regularization} technique with its global convergence guarantees~\cite{nesterov2006cubic}.
	Adaptive and inexact second-order methods
	were developed in~\cite{cartis2011adaptive1,cartis2011adaptive2}.
	Universal algorithms based on the idea of cubic regularization suitable to minimize functions with H\"older continuous Hessian were considered in~\cite{grapiglia2017regularized,doikov2021minimizing}.
	
	Some of the further successful globalization approaches to Newton's method in the convex case include
	\textit{contracting technique}~\cite{doikov2021new},
	\textit{gradient regularization}~\cite{mishchenko2023regularized,doikov2023gradient,doikov2024super}.
	Notably, the Newton method with gradient regularization possesses the 
	same global $O(1 / k^2)$-rate as the cubically regularized Newton
	on convex functions with Lipschitz continuous Hessian~(see Table~\ref{tab:algorithms}).
	At the same time, this method requires solving only one linear system per iteration, using as the regularization parameter $\alpha_k$ a power of the gradient norm at the current iteration.
	In contrast, the cubically regularized Newton method requires solving a univariate nonlinear subproblem of the same complexity as that for trust-region methods~\cite{conn2000trust}.

	A parallel line of research is devoted to \textit{acceleration of the Newton method},
	with the core initial result of the accelerated cubic Newton~\cite{nesterov2008accelerating}.
	The latter algorithm requires performing one cubic Newton step and one correcting gradient step,
	 which is an auxiliary minimization of the estimating function.
	Thus, the complexity of each step is of the same cost as for the cubic Newton,
	while the method achieves an improved $O(1 / k^3)$-rate on our problem class
	of convex functions with Lipschitz Hessian. 
	
	This rate was improved in~\cite{monteiro2013accelerated},
	by replacing the cubic prox function by quadratic one and using inexact proximal-point steps
	with the total rate of $O(1 / k^{7/2})$.
	However, each iteration would require solving an auxiliary line-search subproblem with the total logarithmic cost
	of second-order oracle calls and non-trivial approximate proximal-point steps.
	The latter rate was shown to be the optimal one, matching the corresponding lower bounds~\cite{arjevani2019oracle,agarwal2018lower}. 
	Further, using the language of high-order proximal-point steps, it was shown in~\cite{nesterov2023inexact}
	that the corresponding nonlinear subproblem required to be solved in each iteration of the optimal methods can be made convex, albeit still nonlinear. The optimal second-order scheme from~\cite{monteiro2013accelerated}
	was extended to an arbitrary non-Euclidean geometry in~\cite{contreras2024non}.
	
    Notice that the gap between the optimal complexity $O(1 / \varepsilon^{2/7})$ and the corresponding  $O(1 / \varepsilon^{1/3})$ of the accelerated cubic Newton from~\cite{nesterov2008accelerating} 
	is of the factor $1 / \varepsilon^{(1/3 - 2/7)} = 1 / \varepsilon^{1 / 21}$, which seems rather negligible for practical applications
	and, with a careless implementation, can be offset by the cost of the subproblem.
	
	An optimal in terms of the oracle complexity second-order method, omitting an auxiliary logarithm, was developed in~\cite{kovalev2022first}.
	The accelerated scheme in~\cite{kovalev2022first} is based on inexact proximal-point view
	with a \textit{predefined} optimal rate of convergence, which omits solving an auxiliary nonlinear line-search subproblem as in~\cite{monteiro2013accelerated}.
	At the same time, each iteration might require performing several
	cubically regularized second-order steps, with amortized optimal complexity.
	Another approach for obtaining an optimal second-order oracle complexity was proposed in~\cite{carmon2022optimal} and based on the idea of \textit{adaptive} Monteiro-Svaiter acceleration from~\cite{monteiro2013accelerated}. Note that in both optimal methods from~\cite{kovalev2022first} and~\cite{carmon2022optimal}, each step of the accelerated scheme requires solving a certain amortized number of linear systems, as well as performing an auxiliary \textit{extragradient step}.
	
	The idea of the extragradient corrections was further successfully employed in the Extra-Newton method from~\cite{antonakopoulos2022extra}, obtaining an accelerated $O(1/k^3)$ rate by performing
	minimization of the quadratic model over the bounded convex set followed by an extragradient correction.
	In the simplest case of the feasible set provided by the Euclidean ball, the subproblem from~\cite{antonakopoulos2022extra} is of trust-region type and its cost is the same as that one of the cubically regularized Newton~\cite{nesterov2008accelerating}. At the same time, the method~\cite{antonakopoulos2022extra}
	admits the use of stochastic gradients and Hessians.
	The complexity of the accelerated stochastic second-order methods was improved in~\cite{agafonov2023advancing}, and for the inexact second-order oracles in~\cite{chen2026optimal} based on the inexact proximal-point framework~\cite{monteiro2013accelerated}. 
	
	Therefore, all of the aforementioned acceleration schemes perform an extragradient correction (or an auxiliary estimating function minimization), in order to shift the new proximal center against possible errors caused by 
	approximately solving a certain nonlinear subproblem.
	
	Finally, we note that an independent line of work considers more refined problem classes, such as 
	convex functions with Lipschitz or H\"older-continuous \textit{third-derivative}~\cite{grapiglia2020tensor,doikov2024super,hanzely2026newton}, \textit{quasi-self-concordant} and \textit{generalized-self-concordant} 
	function~\cite{bach2010self,sun2019generalized,carmon2020acceleration,doikov2025minimizing,semenov2026gradient}, for which it is possible to obtain faster rate beyond 
	those presented in Table~\ref{tab:algorithms}.
	Such acceleration comes from more refined smoothness assumptions on the target objective.
	In this work, we focus on the broad class of convex functions with Lipschitz Hessian,
	for which our new algorithmic framework seems to be substantially new.
	
	\paragraph{Contributions.}

	We propose a very simple second-order acceleration algorithm~\eqref{FNM},
	which, in contrast to previously known schemes,
	operates solely in the \textit{primal space of variables} (without extragradient corrections). 
	
	Our new method can be interpreted as a further development of contracting-proximal acceleration framework from~\cite{doikov2020contracting}. In contrast to~\cite{doikov2020contracting}, where each iteration requires approximately minimizing an auxiliary \textit{contracting subproblem} in the logarithmic number of steps of the basic method, we show that it is enough to perform only \textit{a single Newton step} to ensure that the new point is sufficiently close to the minimizer of the contracted subproblem. In this regard, our viewpoint resembles the classic path-following approach from the theory of the interior-point methods~\cite{renegar2001mathematical,nesterov1994interior,nesterov2018lectures,dvurechensky2018global,dvurechensky2025improved}.

	Another interpretation of our approach is a direct generalization of the fast gradient method in the classic momentum-based form~\eqref{FGM}. 
	
	Our new method uses only one linear system solve per iteration, 
	and does not require solving any nonlinear subproblem (as in cubic regularization), performing auxiliary searches, or performing extragradient corrections. At the same time, it achieves the accelerated $O(1/k^3)$ rate as the original acceleration of the cubic Newton method from~\cite{nesterov2008accelerating}. Therefore, to the best of our knowledge, this is the first accelerated second-order method with such a simple implementation, achieving such rate.
	
	In Section~\ref{SectionReparametrization} we show how the momentum-based iterations~\eqref{FNM} are related to the more standard prox-center notation used in the literature.
	
	For our analysis, the core result about the Newton method which we use is its classic \textit{local quadratic convergence}, which we briefly review in Section~\ref{SectionLocal}. 
	
	We provide the accelerated analysis of our algorithm and prove the $O(1/k^3)$ rate in Section~\ref{SectionAcceleration}.
	
	In Section~\ref{SectionBregman}, we generalize our construction further, showing that we can implement our method
	for different geometries through the language of Bregman divergence, as well as for composite optimization problems, including simple constraints or non-smooth regularizers, while the corresponding iteration requires solving a nonlinear subproblem. However, we show the same $O(1 / k^3)$ global rate even with approximate solution to the subproblem.
	In the case of the simplest second-order iterate of the form~\eqref{FNM} our result additionally shows that there is no need to solve the linear system exactly. This allows to implement each iteration with efficient large-scale Hessian-free solvers.

	\section{Algorithm Reparametrization}
	\label{SectionReparametrization}
	
	We analyze convergence of method~\eqref{FNM} with exact Hessian: $\mat{H}_k = \nabla^2 f(\yy_k)$. Therefore, our method would require to solve just one linear system per iteration with the new Hessian computed at the intermediate point. 
	
	We fix the standard Euclidean norm for vectors $\| \xx \|_2 := \la \xx, \xx \ra^{1/2}$, while we 
	show the extension of our analysis to arbitrary norms through Bregman divergence in Section~\ref{SectionBregman}.

	To simplify the reasoning, it is convenient to use the following standard reparameterization.
	We fix an arbitrary sequence of positive numbers $\{ a_k \}_{k \geq 0}$
	and denote their partial sums, starting from $A_0 := 0$, and the normalized ratios by:
	$$
	\ba{rcl}
	A_k & := & \sum\limits_{i = 1}^k a_i, \qquad \gamma_k \;\; := \;\; \frac{a_{k + 1}}{A_{k + 1}} \;\; \in \;\; (0, 1].
	\ea
	$$
	Along with the main sequence of points $\{ \xx_k \}_{k \geq 0}$ we keep track of the auxiliary sequence of prox centers $\{ \vv_k \}_{k \geq 0}$, starting from the same initialization $\vv_0 = \xx_0 \in \R^n$.
	Then, the \textit{prediction points} used to collect the oracle information about the function 
	are taken as convex combinations:
	\beq \label{YkDef}
	\ba{rcl}
	\yy_k & := & \gamma_k \vv_k + (1 - \gamma_k) \xx_k.
	\ea
	\eeq
	Our analysis is inspired by the framework of 
	contracting-point methods~\cite{doikov2020contracting},
	which suggests to compute the next proximal point $\vv_{k + 1}$
	as an inexact minimization of the following subproblem
	with \textit{contracted objective}:
	\beq \label{ContractedSubproblem}
	\ba{rcl}
	\vv_{k + 1} & \approx & \argmin\limits_{\xx \in \R^n}
	\Bigl[  \, h_k(\xx) \;\; := \;\; A_{k + 1} f\bigl( \gamma_k \xx + (1 - \gamma_k) \xx_k \bigr) + \frac{1}{2}\| \xx - \vv_k \|_2^2 \, \Bigr].
	\ea
	\eeq
	Our main observation is that it is enough to perform just a \textit{single classical Newton step}
	to compute the next prox center:
	\beq \label{VkNext}
	\boxed{
	\ba{rcl}
	\vv_{k + 1} & := & \vv_k - \nabla^2 h_k(\vv_k)^{-1} \nabla h_k(\vv_k) \\[10pt]
	& = &
	\vv_k - a_{k + 1} \Bigl(  \frac{a_{k + 1}^2}{A_{k + 1}} \nabla^2 f(\yy_k) + \mat{I} \Bigr)^{-1} \nabla f(\yy_k).
	\ea
	}
	\eeq
	It appears that by controlling $\gamma_k$, we can ensure that the current point $\vv_k$
	is in the \textit{region of local quadratic convergence} of Newton's method, and one step~\eqref{VkNext} is enough to ensure a good approximation to the minimum of the contracted subproblem~\eqref{ContractedSubproblem}. This observation also resembles the classic theory of path-following interior-point methods~\cite{nesterov2018lectures}.
	
	After the prox center is updated, we follow the \textit{rule of similar triangles} to get the next main point:
	\beq \label{XkNext}
	\ba{rcl}
	\xx_{k + 1} & := & \gamma_k \vv_{k + 1} + (1 - \gamma_k) \xx_k.
	\ea
	\eeq
	It is easy to see that $\xx_{k + 1}$ and $\yy_{k}$ are only different in the displacement of the $\vv$-component:
	$$
	\ba{rcl}
	\xx_{k + 1} - \yy_k
	& \overset{(\ref{YkDef}), (\ref{XkNext})}{=} &
	\gamma_k (\vv_{k + 1} - \vv_k)
	\;\; \overset{(\ref{VkNext})}{=} \;\;
	-\Bigl( \nabla^2 f(\yy_k) + \frac{A_{k + 1}}{a_{k + 1}^2} \mat{I} \Bigr)^{-1} \nabla f(\yy_k),
	\ea
	$$
	which gives exactly the first formula in~\eqref{FNM} with the regularization parameter equal to
	\beq \label{AlphaKFormula}
	\boxed{
	\ba{rcl}
	\alpha_k & = & \frac{a_{k + 1}^2}{A_{k + 1}}.
	\ea
	}
	\eeq
	It remains to obtain the momentum formula for the next prediction point $\yy_{k + 1}$. We notice that
	$$
	\ba{rcl}
	\yy_{k + 1}
	& := &
	\gamma_{k + 1} \vv_{k + 1} + (1 - \gamma_{k + 1}) \xx_{k + 1}
	\;\; \overset{(\ref{XkNext})}{=} \;\;
	\gamma_{k + 1} \Bigl[ 
	\xx_k + \frac{1}{\gamma_k}(\xx_{k + 1} - \xx_k)
	\Bigr] + (1 - \gamma_{k + 1}) \xx_{k + 1} \\
	\\
	& = &
	\xx_{k + 1} 
	+ \gamma_{k + 1} \Bigl[ \frac{1}{\gamma_k} - 1 \Bigr] (\xx_{k + 1} - \xx_k),
	\ea
	$$
	which gives the prediction update in~\eqref{FNM} with the momentum parameter equal to
	\beq \label{BetaKFormula}
	\boxed{
	\ba{rcl}
	\beta_k & = & \gamma_{k + 1} \Bigl[ \frac{1}{\gamma_k} - 1 \Bigr]
	\;\; = \;\;
	\frac{a_{k + 2}}{a_{k + 1}} \cdot \frac{A_k}{A_{k + 2}}.
	\ea
	}
	\eeq
	Therefore, we have a direct correspondence between our new parameterization and the primal form of method~\eqref{FNM}.

	\section{Properties of One Newton's Step}
	\label{SectionLocal}

	Our assumption on the objective $f$ is that it is convex and it has the Lipschitz continuous Hessian with some constant $L_2 > 0$:
	\beq \label{LipHessEuclidean}
	\ba{rcl}
	\| \nabla^2 f(\yy) - \nabla^2 f(\xx) \| & \leq & L_2 \|\yy - \xx\|_2, \qquad \xx, \yy \in \R^n,
	\ea
	\eeq
	where in the left hand side of~\eqref{LipHessEuclidean} we have the standard spectral norm of a symmetric matrix.
	Thus,
	\beq \label{LipHess}
	\ba{rcl}
	\| \nabla f(\yy) - \nabla f(\xx) - \nabla^2 f(\xx)(\yy - \xx) \|_2 & \leq & \frac{L_2}{2}\| \yy - \xx \|_2^2,
	\qquad \xx, \yy \in \R^n.
	\ea
	\eeq
	
	Let us investigate what we can say about one Newton's step~\eqref{VkNext} as applied to the contracted objective $h_k(\cdot)$, defined by~\eqref{ContractedSubproblem}. We denote by $\vv_{k}^{\star}$ the exact minimum of $h_k$:
	$$
	\ba{rcl}
	\vv_{k}^{\star} & := & \argmin\limits_{\xx \in \R^n} h_k(\xx),
	\ea
	$$
	which is unique due to strong convexity of the regularizer; we have $\nabla^2 h_k(\vv) \succeq \mat{I}$ for all $\vv \in \R^n$. The optimum satisfies the following stationary condition:
	\beq \label{VkStarStat}
	\ba{rcl}
	\nabla h_k(\vv_{k}^{\star})
	& = &
	a_{k + 1} \nabla f(\xx_{k + 1}^{\star}) + \vv_{k}^{\star} - \vv_k
 	\;\; = \;\; \0, \\
 	\\
	\xx_{k + 1}^{\star} & := & \gamma_k \vv_{k}^{\star} + (1 - \gamma_k) \xx_k.
	\ea
	\eeq
	Therefore, we can relate the distance to the optimum $\vv_{k}^{\star}$ after one Newton's step, as follows:
	\beq \label{NextIterDistance}
	\ba{rcl}
	\| \vv_{k + 1} - \vv_{k}^{\star} \|_2
	&  \leq  & 
	\| \nabla^2 h_k(\vv_k) (\vv_{k + 1} - \vv_{k}^{\star}) \|_2 \\
	\\
	& \overset{(\ref{VkStarStat}), (\ref{VkNext})}{=} &
	\| \nabla h_k(\vv_{k}^{\star}) - \nabla h_k(\vv_k) - \nabla^2 h_k(\vv_k)(\vv_{k}^{\star} - \vv_k)   \|_2 \\
	\\
	& \overset{(\ref{LipHess})}{\leq} &
	\frac{\ell_{k}}{2}\| \vv_k - \vv_{k}^{\star}\|^2_2,
	\ea
	\eeq
	where 
	$$
	\ba{rcl}
	\ell_{k} & := & \frac{a_{k + 1}^3}{A_{k + 1}^2} L_2
	\ea
	$$
	is the Lipschitz constant of the Hessian of the contracted objective.
	
	Note that our simple reasoning~\eqref{NextIterDistance} establishes the classic result
	about local quadratic convergence of pure Newton's method. 
	The sufficient condition for the contraction to ensure that the current point $\vv_k$
	is in the region of local quadratic convergence is
	\beq \label{RegionLocalConvergence}
	\ba{rcl}
	\| \vv_k - \vv_{k}^{\star} \|_2 & \leq & \frac{2}{\ell_k},
	\ea
	\eeq
	which we can control by choosing an appropriately small $\ell_k$. This choice defines the convergence rate for our accelerated method, as we see in the next section.
	
	Together with the distance to the solution, we also need the bound for the new gradient norm:
	\beq \label{GradNormBound}
	\ba{rcl}
	\!\!\!\!\!
	\| \nabla h_k(\vv_{k + 1}) \|_2 & \! \overset{(\ref{VkNext})}{=} \! &
	\| \nabla h_k(\vv_{k + 1}) - \nabla h_k(\vv_k) - \nabla^2 h_k(\vv_k)(\vv_{k + 1} - \vv_k) \|_2
	\overset{(\ref{LipHess})}{\leq} 
	\frac{\ell_k}{2}\| \vv_{k + 1} - \vv_k \|_2^2.
	\ea
	\eeq
	Coupling this bound with an upper bound on the step length: $\| \vv_{k + 1} - \vv_k \|_2 \leq \| \nabla h_k(\vv_k) \|_2$, we have the quadratic convergence also in terms of the gradient norm.

	\section{Accelerated Rate Analysis}
	\label{SectionAcceleration}
	
	We are ready to establish the following main convergence result.
	
	\begin{theorem}
		\label{TheoremMain}
		Assume that a solution~$x^{\star}$ to~\eqref{MainProblem} exists and let $R \geq \|  \xx_0 - \xx^{\star} \|_2$ be any upper estimate.
		Let $a_{k + 1}$ be chosen such that
		\beq \label{aKChoice}
		\ba{rcl}
		\frac{a_{k + 1}^3}{A_{k + 1}^2} & \leq & \frac{1}{L_2 R}.
		\ea
		\eeq
		Then, the following invariant is preserved, for any $k \geq 0$:
		\beq \label{MainInvariant}
		\ba{rcl}
		\frac{1}{2}\| \xx_0 - \xx^{\star} \|^2_2
		+ A_k f(\xx^{\star}) & \geq & 
		\frac{1}{2} \| \vv_k  - \xx^{\star} \|^2_2
		+ A_k f(\xx_k).
		\ea
		\eeq
	\end{theorem}
	\begin{proof}
		We prove~\eqref{MainInvariant} by induction. For $k = 0$ it holds as we set $A_0 := 0$ and $\xx_0 = \vv_0$.
		
		Assume that~\eqref{MainInvariant} holds for a current iterate $k \geq 0$. Then, for the next step:
		\beq \label{ReasoningBeg}
		\ba{rcl}
		\frac{1}{2}\| \xx_0 - \xx^{\star} \|_2^2 + A_{k + 1} f(\xx^{\star})
		& = & 
		\frac{1}{2}\| \xx_0 - \xx^{\star} \|_2^2 + A_{k} f(\xx^{\star}) + a_{k + 1}f(\xx^{\star}) \\
		\\
		& \overset{(\ref{MainInvariant})}{\geq} & 
		\frac{1}{2} \| \vv_k - \xx^{\star} \|_2^2 + A_{k} f(\xx_k) + a_{k + 1}f(\xx^{\star}).
		\ea
		\eeq
		By convexity of $f$, the right hand side can be lower bounded by the contracted objective $h_k(\cdot)$ at~$\xx^{\star}$, which is strongly convex. Thus,
		$$
		\ba{rcl}
		\frac{1}{2} \| \xx_0 - \xx^{\star} \|_2^2 + A_{k + 1} f(\xx^{\star}) & \geq & h_k(\xx^{\star})
		\;\; \geq \;\; 
		\frac{1}{2} \| \vv_{k}^{\star} - \xx^{\star} \|_2^2 + h_k(\vv_k^{\star}) \\
		\\
		& = & 
		\frac{1}{2} \| \vv_{k}^{\star} - \xx^{\star} \|_2^2 + 
		\frac{1}{2} \| \vv_{k}^{\star} - \vv_k \|_2^2 + A_{k + 1} f(\xx_{k + 1}^{\star}) \\
		\\
		& \geq &
		\frac{1}{2} \| \vv_{k}^{\star} - \xx^{\star} \|_2^2 + 
		\frac{1}{2} \| \vv_{k}^{\star} - \vv_k \|_2^2 + A_{k + 1} f(\xx^{\star}).
		\ea
		$$
		We obtain the following useful bound on the distance to the exact minimum~$\vv_k^{\star}$ of the contracted objective:
		\beq \label{BoundContractedMin}
		\ba{rcl}
		\| \vv_k^{\star} - \xx^{\star} \|_2^2 + 
		\| \vv_k^{\star} - \vv_k \|_2^2 & \leq & \| \xx_0 - \xx^{\star}\|^2 \;\; \leq \;\; R^2.
		\ea
		\eeq
		However, in the actual algorithm we do not have access to the exact minimum $\vv_k^{\star}$. Instead, we substitute for it one Newton's step $\vv_{k + 1}$. Applying the convexity lower bound for both function values in~\eqref{ReasoningBeg},
		$$
		\ba{rcl}
		f(\xx_k) & \geq & f(\xx_{k + 1}) + \la \nabla f(\xx_{k + 1}), \xx_k - \xx_{k + 1} \ra, \\
		\\
		f(\xx^{\star}) & \geq & f(\xx_{k + 1}) + \la \nabla f(\xx_{k + 1}), \xx^{\star} - \xx_{k + 1} \ra,
		\ea
		$$
		we get:
		$$
		\ba{rcl}
		& & \!\!\!\!\!\!\!\!\!\!\!\!
		\frac{1}{2}\| \xx_0 - \xx^{\star} \|_2^2 + A_{k + 1} f(\xx^{\star}) \\
		\\
		& \geq & 
		\frac{1}{2}\| \vv_k - \xx^{\star} \|_2^2 + A_{k + 1} \Bigl[ f(\xx_{k + 1})
		+  \la \nabla f(\xx_{k + 1}),  (\gamma_k \xx^{\star} + (1 - \gamma_k) \xx_k) - \xx_{k + 1} \ra \Bigr] \\
		\\
		& = & 
		\frac{1}{2}\| \vv_k - \xx^{\star} \|_2^2 + A_{k + 1} f(\xx_{k + 1}) + a_{k + 1} \la \nabla f(\xx_{k + 1}),  \xx^{\star} - \vv_{k + 1} \ra.
		\ea
		$$
		Since the norm is Euclidean, we can expand the square as follows:
		\beq \label{EuclideanSquare}
		\ba{rcl}
		\frac{1}{2}\| \vv_k - \xx^{\star} \|_2^2
		& = &
		\frac{1}{2} \| \vv_{k + 1} - \xx^{\star} \|_2^2
		+ \frac{1}{2} \| \vv_k - \vv_{k + 1} \|_2^2
		+ \la \vv_k - \vv_{k + 1}, \vv_{k + 1} - \xx^{\star} \ra.
		\ea
		\eeq
		Substituting this equation we finally arrive at
		$$
		\ba{rcl}
		& & \!\!\!\!\!\!\!\!\!\!\!\!
		\frac{1}{2}\| \xx_0 - \xx^{\star} \|_2^2 + A_{k + 1} f(\xx^{\star}) \\
		\\
		& \geq & 
		\frac{1}{2}\| \vv_{k + 1} - \xx^{\star} \|_2^2 + A_{k + 1} f(\xx_{k + 1}) + \frac{1}{2} \| \vv_{k} - \vv_{k + 1} \|_2^2
		+ \la \vv_{k + 1} - \vv_{k} + a_{k + 1} \nabla f(\xx_{k + 1}), \xx^{\star} - \vv_{k + 1} \ra \\
		\\
		& = & 
		\frac{1}{2}\| \vv_{k + 1} - \xx^{\star} \|_2^2 + A_{k + 1} f(\xx_{k + 1})  + \frac{1}{2} \| \vv_{k} - \vv_{k + 1} \|_2^2
		+ \la \nabla h_k(\vv_{k + 1}), \xx^{\star} - \vv_{k + 1} \ra \\
		\\
		& \overset{(\ref{GradNormBound})}{\geq} & 
		\frac{1}{2}\| \vv_{k + 1} - \xx^{\star} \|_2^2 + A_{k + 1} f(\xx_{k + 1})
		+ \frac{1}{2} \| \vv_{k} - \vv_{k + 1} \|_2^2 \cdot \bigl( 1 - \ell_k \| \xx^{\star} - \vv_{k + 1} \|_2 \bigr),
		\ea
		$$
		and to establish~\eqref{MainInvariant} for the next iterate, it is sufficient to ensure
		$\ell_k \| \xx^{\star} - \vv_{k + 1} \|_2 \leq 1$, which follows from our choice
		\beq \label{EllKChoice}
		\ba{rcl}
		\ell_k R & = & \frac{a_{k + 1}^3}{A_{k + 1}^2} L_2 R \;\; \leq \;\; 1,
		\ea
		\eeq
		and the following bound:
		$$
		\ba{rcl}
		\ell_k \| \xx^{\star} - \vv_{k + 1} \|_2
		& \leq &
		\ell_k \bigl( \, \| \xx^{\star} - \vv_{k}^{\star} \|_2 + \| \vv_{k + 1} - \vv_k^{\star} \| \, \bigr) \\
		\\
		& \overset{(\ref{NextIterDistance})}{\leq} &
		\ell_k \|\xx^{\star} - \vv_k^{\star} \|_2 + \frac{\ell_k^2}{2} \| \vv_k - \vv_k^{\star} \|_2^2 \\
		\\
		& \overset{(\ref{BoundContractedMin})}{\leq} &
		\max\limits_{\substack{(a, b) \in \R^2 \\ a^2 + b^2 \leq R^2}}
		\Bigl[ \ell_k a + \frac{\ell_k^2}{2} b^2 \Bigr]
		\;\; \overset{(\ref{EllKChoice})}{\leq} \;\; 1.
		\ea
		$$
		Therefore, we proved~\eqref{MainInvariant} for all $k \geq 0$.
	\end{proof}
	
	From~\eqref{MainInvariant}, we immediately obtain the global convergence rate:
	$f(\xx_k) - f^{\star} \leq \frac{R^2}{2 A_k}$.
	It remains to choose the controlling coefficients $\{ a_{k} \}_{k \geq 1}$ such that~\eqref{aKChoice} is satisfied and that $A_k$ is as big as possible.
	Among various possibilities, we can use the following predetermined choice that leads to simple formulas for our momentum update~\eqref{FNM}. We set
	$$
	\ba{rcl}
	A_{k + 1} & = & \frac{1}{3^3 L_2 R} (k + 1)(k + 2)(k + 3),
	\ea
	$$
	and, hence,
	$$
	\ba{rcl}
	a_{k + 1} \;\; = \;\; A_{k + 1} - A_{k}
	& = &
	\frac{1}{3^3 L_2R} (k + 1)(k + 2) (k + 3 - k)
	\;\; = \;\;
	\frac{1}{3^2 L_2R} (k + 1)(k + 2).
	\ea
	$$
	This choice ensures that condition~\eqref{aKChoice} holds:
	$$
	\ba{rcl}
	\frac{a_{k + 1}^3}{A_{k + 1}^2}
	& = & 
	\frac{1}{L_2 R} \cdot \frac{(k + 1)^3 (k + 2)^3}{(k + 1)^2 (k + 2)^2 (k + 3)^2}
	\;\; = \;\;
	\frac{1}{L_2 R} \cdot \frac{(k + 1)(k + 2)}{(k + 3)^2}
	\;\; \leq \;\;
	\frac{1}{L_2 R}, \qquad k \geq 0.
	\ea
	$$
	At the same time, we have
	\beq \label{AlphaBetaK}
	\ba{rcl}
	\alpha_k & \overset{(\ref{AlphaKFormula})}{=} &
	\frac{a_{k + 1}^2}{A_{k + 1}}
	\;\; = \;\;
	\frac{1}{3L_2 R} \cdot \frac{(k + 1) (k + 2)}{(k + 3)} \\
	\\
	\beta_k & \overset{(\ref{BetaKFormula})}{=} &
	\frac{a_{k + 2} A_k}{a_{k + 1} A_{k + 2}}
	\;\; = \;\;
	\frac{(k + 2)(k + 3) k (k + 1) (k + 2)}{(k + 1)(k + 2) (k + 2) (k + 3) (k + 4)}
	\;\; = \;\;
	\frac{k}{k + 4}.
	\ea
	\eeq
	We obtain the following concrete instance of the primal accelerated Newton's method:
	\beq \label{PrimalNewton}
	\boxed{
	\left.
	\ba{rcl} 
	\xx_{k + 1} & = & \yy_k - \bigl( \nabla^2 f(\yy_k) + 3L_2 R \frac{k + 3}{(k + 1)(k + 2)} \mat{I} \bigr)^{-1} \nabla f(\yy_k) \\
	\\
	\yy_{k + 1} & = & \xx_{k + 1} + \frac{k}{k + 4} (\xx_{k + 1} - \xx_k)
	\ea \right\},  \quad k \geq 0.
	}
	\eeq
	for which we have the following global rate.
	\begin{corollary}
		Let $R \geq \| \xx_0 - \xx^{\star} \|_2$ in method~\eqref{PrimalNewton}. Then, we have:
		$$
		\ba{rcl}
		f(\xx_k) - f^{\star} & \leq & \frac{3^3 L_2 R^3}{2k (k + 1) (k + 2)}
		\;\; = \;\;
		O\Bigl( \frac{L_2 R^3}{k^3} \Bigr), \qquad k \geq 1.
		\ea
		$$
	\end{corollary}
	
	In practice, instead of the rigid choice of parameter $L_2R$ in~\eqref{PrimalNewton}, we might want to choose
	a tuned second-order stepsize to have the fastest rate for a particular problem.
	
	At the same time, from a theoretical perspective, optimization algorithm~\eqref{PrimalNewton}
	is the first known to us second-order method with only one linear system solve per iteration and with the global
	$O(1/k^3)$ rate of convergence for our problem class of convex functions with Lipschitz continuous Hessian.
	
	Instead of using exact solution to the linear system in~\eqref{PrimalNewton} at each iteration,
	it is sufficient to have an approximate solution, which we show in the next section.
	We also further extend method~\eqref{PrimalNewton} for solving \textit{composite} optimization problem (allowing to use simple constraints or a non-smooth regularizer in~\eqref{MainProblem}), as well as allow the method to adapt
	to a different geometry of the problem through Bregman divergence.

	\section{Bregman Geometry and Composite Problems}
	\label{SectionBregman}
	
	We show that our approach can be easily generalized to more abstract settings.
	
	Let $\V$ be a finite-dimensional real vector space with a fixed norm $\| \cdot \|$ on it (not necessarily Euclidean).
	We denote 
	by $\V^{*}$ the dual space, which is the space of linear functions on $\V$.
	We use the symbol $\la \cdot, \cdot \ra$ to denote the value
	of a linear form $\ss \in \V^{*}$ on a vector $\xx \in \V$:
	$$
	\ba{rcl}
	\la \ss, \xx \ra & := & \ss(\xx) \;\; \in \;\; \R.
	\ea
	$$
	The dual norm is defined in the standard way:
	$$
	\ba{rcl}
	\| \ss \|_* & := & \max\limits_{\xx \in \V \; : \; \| \xx \| \leq 1} \la \ss, \xx \ra, \qquad \ss \in \V^{*}.
	\ea
	$$
	
	We consider optimization problems in the following \textit{composite form}:
	\beq \label{MainComposite}
	\ba{rcl}
	\min\limits_{\xx \in Q} \Bigl[ \, F(\xx) & := & f(\xx) + \psi(\xx) \, \Bigr],
	\ea
	\eeq
	where $\psi$ is a \textit{simple} regularizer, $\dom \psi = Q \subseteq \V$,
	which is a proper closed convex function, possibly non-differentiable. At the same time $f$ is several times differentiable in a neighborhood of every $\xx \in Q$. We treat its gradient as the linear form, $\nabla f(\xx) \in \V^{*}$, and the Hessian is the linear map $\nabla^2 f(\xx): \V \to \V^{*}$,
	which we assume to be Lipschitz continuous on $Q$ with constant $L_2 > 0$:
	\beq \label{LipHessGeneral}
	\ba{rcl}
	\| \nabla^2 f(\yy) - \nabla^2 f(\xx) \|
	& := & 
	\max\limits_{\hh \in \V \; : \; \| \hh \| \leq 1} \| (\nabla^2 f(\yy) - \nabla^2 f(\xx)) \hh \|_* \\
	\\
	& \leq &
	L_2 \|\yy - \xx \|, \qquad \xx, \yy \in Q.
	\ea
	\eeq
	The standard integration argument provides us the bound for the gradient approximation:
	\beq \label{GradApprox}
	\ba{rcl}
	\| \nabla f(\yy) - \nabla f(\xx) - \nabla^2 f(\xx)(\yy - \xx) \|_* & \leq & \frac{L_2}{2}\| \yy - \xx \|^2, 
	\qquad \xx, \yy \in Q.
	\ea
	\eeq
	
	We also fix a $1$-strongly convex differentiable \textit{distance function} $d(\cdot)$.
	The Bregman divergence is:
	\beq \label{Bregman}
	\ba{rcl}
	\beta_d(\xx; \yy) & := & d(\yy) - d(\xx) - \la \nabla d(\xx), \yy - \xx \ra
	\;\; \geq \;\;
	\frac{1}{2}\| \yy - \xx \|^2.
	\ea
	\eeq
	
	The main algebraic identity which is the key structural property of the Bregman divergence (compare it with~\eqref{EuclideanSquare}) is the next one, for any $\vv$, $\vv^+$, and $\xx$ it holds
	\beq \label{BregmanMain}
	\ba{rcl}
	\beta_d(\vv; \xx)
	& = & 
	\beta_d(\vv^+; \xx) + \beta_d(\vv; \vv^+) + \la \nabla d(\vv) - \nabla d(\vv^+), \vv^+ - \xx \ra,
	\ea
	\eeq	
	which follows directly from the definition.
	
	Consider the following general second-order scheme, starting from $\xx_0 = \vv_0 \in Q$,
	iterate $k \geq 0$:
	\begin{tcolorbox}[colback=white, colframe=black, sharp corners, boxrule=0.5pt]
	\begin{enumerate}
		\item Set $\yy_k := \gamma_k \vv_k + (1 - \gamma_k) \xx_k$ and use $\mat{H}_k = \nabla^2 f(\yy_k)$.
		\item Compute $\vv_{k + 1}$ such that
		$$
		\ba{rcl}
		\frac{1}{a_{k + 1}} \bigl( \nabla d(\vv_{k + 1}) - \nabla d(\vv_k) \bigr)
		+
		\nabla f(\yy_k)
		+
		\gamma_k \mat{H}_k \bigl( \vv_{k + 1} - \vv_k \bigr)
		+ \psi'(\vv_{k + 1}) & \approx & \0 \;\; \in \;\; \V^{*},
		\ea
		$$
		where $\psi'(\vv_{k + 1}) \in \partial \psi(\vv_{k + 1})$ is a subgradient of the regularizer.
		\item Set $\xx_{k + 1} := \gamma_k \vv_{k + 1} + (1 - \gamma_k) \xx_k$.
	\end{enumerate}
\end{tcolorbox}

	Note that computation of $\vv_{k + 1}$ corresponds to the approximate minimization of the following subproblem:
	\beq \label{VkMinBregman}
	\ba{rcl}
	\vv_{k + 1} & \approx & \argmin\limits_{\xx \in Q}
	\Bigl\{ 
	\;
	a_{k + 1} \la \nabla f(\yy_{k}), \xx - \vv_k \ra
	+ \frac{a_{k + 1}^2}{2 A_{k + 1}} \la \mat{H}_k(\xx - \vv_k), \xx - \vv_k \ra \\[10pt]
	& & \qquad \qquad \quad \;\;
	+ \; a_{k + 1}\psi(\xx) + \beta_d(\vv_k; \xx)
	\;
	\Bigr\}.
	\ea
	\eeq
	This is minimization of a convex quadratic function, regularized by the Bregman divergence, and augmented by the composite component. Hence, we assume that $\psi$ is simple enough that we can perform this minimization efficiently.
	Moreover, we need only an \textit{inexact solution to~\eqref{VkMinBregman}} 
	with the following simple tolerance condition. We require to compute $\vv_{k + 1}$ which satisfies
	the inequality:
	\beq \label{DeltaKDef}
	\ba{rcl}
	\begin{cases}
	& \big\|
	\frac{1}{a_{k + 1}} \bigl( \nabla d(\vv_{k + 1}) - \nabla d(\vv_k)  \bigr)
	+ \nabla f(\yy_k) + \gamma_k \mat{H}_k(\vv_{k + 1} - \vv_k) + \psi'(\vv_{k + 1})
	\bigr\|_*
	\; \leq \;
	\delta_{k + 1}, \\[15pt]
	& 
	\delta_{k + 1} \; := \; \frac{\gamma_k^2 L_2}{2} \min\Bigl[\, \| \vv_{k + 1} - \vv_k \|^2, \, R^2  \, \Bigr],
	\end{cases}
	\ea
	\eeq
	where $\psi'(\vv_{k + 1}) \in \partial \psi(\vv_{k + 1})$ is a certain subgradient, and $R > 0$ is a fixed parameter. Note that we can expect from any reasonable numerical method
	solving subproblem~\eqref{VkMinBregman} to make the left hand side in the inequality in~\eqref{DeltaKDef}
	small, and our condition with $\delta_{k + 1}$ just specifies the required tolerance precisely.
	Note that it is very important that we admit inexact solution to~\eqref{VkMinBregman}, as in general,
	this subproblem can be quite complicated due to an arbitrary choice of $d$ and even for simple regularizers $\psi$.
	At the same time, the subproblem remains convex, so standard convex optimization techniques can be employed.
	
	Clearly, without composite part $\psi(\xx) \equiv 0$ and using the Euclidean norm $d(\xx) = \frac{1}{2}\| \xx \|_2^2$, 
	this algorithm gives our previous method with a possibility of using \textit{inexact linear solver}. The condition~\eqref{DeltaKDef} becomes:
	\beq \label{DeltaEuclidean}
	\ba{rcl}
	\big\| 
	\nabla f(\yy_k) + \bigl( \gamma_k \mat{H}_k + \frac{1}{a_{k + 1}} \mat{I} \bigr)(\vv_{k + 1} - \vv_k)
 	\big\|_2 & \leq & 
 	\frac{\gamma_k^2 L_2}{2} \min\Bigl[ \, \| \vv_{k + 1} - \vv_k \|_2^2, \; R^2 \,\Bigr],
	\ea
	\eeq
	and thus we can apply any Hessian-free method solving the linear system, e.g., the conjugate gradient method, making the left hand side in~\eqref{DeltaEuclidean} small. In this case, we do not need to compute the full Hessian. In our solver, we can use only the Hessian-vector products, which is suitable for large-scale models.
	
	For our new most general scheme we prove the following convergence result.
	
	\begin{theorem}
		Let $R \geq \sqrt{2 \beta_d(\xx_0; \xx^{\star})}$
		and let each $\vv_{k + 1}$ satisfy the inexact condition~\eqref{DeltaKDef}.  
		Let $a_{k + 1}$ be chosen such that
		\beq \label{NewAkChoice}
		\ba{rcl}
		\frac{a_{k + 1}^3}{A_{k + 1}^2} & \leq & 
		\frac{2}{5 L_2 R}.
		\ea
		\eeq
		Then, for any $k \geq 0$ we have
		\beq \label{NewInvariant}
		\ba{rcl}
		\beta_d(\xx_0; \xx^{\star}) + A_k F(\xx^{\star}) & \geq & \beta_d(\vv_k; \xx^{\star}) + A_k F(\xx_k).
		\ea
		\eeq
	\end{theorem}
	\begin{proof}
		We prove~\eqref{NewInvariant} by induction in $k \geq 0$. It clearly holds for $k = 0$. Consider any $k \geq 0$:
		\beq \label{GenProof1}
		\ba{rcl}
		\beta_d(\xx_0; \xx^{\star}) + A_{k + 1} F(\xx^{\star})
		& = & 
		\beta_d(\xx_0; \xx^{\star}) + A_k F(\xx^{\star}) + a_{k + 1} F(\xx^{\star}) \\
		\\
		& \overset{(\ref{NewInvariant})}{\geq} &
		\beta_d(\vv_k; \xx^{\star}) + A_k F(\xx_k) + a_{k + 1} F(\xx^{\star}).
		\ea
		\eeq
		The right hand side is lower bounded by the following objective with contracted smooth part:
		$$
		\ba{rcl}
		h_k(\xx) & = &
		A_{k + 1} f\bigl( \gamma_k \xx + (1 - \gamma_k) \xx_k \bigr)
		+ a_{k + 1} \psi(\xx) + A_k \psi(\xx_k) + \beta_d(\vv_k; \xx).
		\ea
		$$
		Denote the minimum of $h_k(\cdot)$ by $\vv_k^{\star}$, which satisfies the optimality condition
		(compare with~\eqref{VkStarStat}):
		\beq \label{BregmanOptimality}
		\ba{rcl}
		-a_{k + 1}\psi'(\vv_k^{\star})
		& := & 
		a_{k + 1} \nabla f(\xx_{k + 1}^{\star})
		+ \nabla d(\vv_k^{\star}) - \nabla d(\vv_k) 
		\;\; \in \;\; -a_{k + 1} \partial \psi(\vv_k^{\star}), \\
		\\
		\xx_{k + 1}^{\star} & := & \gamma_k \vv_k^{\star} + (1 - \gamma_k) \xx_k.
		\ea
		\eeq
		Using the key property of the Bregman divergence~\eqref{BregmanMain} and convexity of all components, we thus have
		$$
		\ba{rcl}
		\beta_d(\xx_0; \xx^{\star}) + A_{k + 1} F(\xx^{\star})
		& \geq & 
		h_k(\xx^{\star})
		\;\; \geq \;\;
		\beta_d(\vv_k^{\star}; \xx^{\star})
		+
		\hh_k(\vv_k^{\star}) \\
		\\
		& \geq & A_{k + 1} F(\xx_{k + 1}^{\star}) + \beta_d(\vv_k; \vv_k^{\star}) + \beta_d(\vv_k^{\star}; \xx^{\star}),
		\ea
		$$
		where we use the same notation as before, $\xx_{k + 1}^{\star} := \gamma_k \vv_k^{\star} + (1 - \gamma_k) \xx_k$. Hence, we obtain the following generalization of~\eqref{BoundContractedMin}:
		\beq \label{BoundGenMin}
		\ba{rcl}
		\beta_d(\vv_k^{\star}; \xx^{\star}) + \beta_d(\vv_k; \vv_k^{\star})
		& \leq & \beta_d(\xx_0; \xx^{\star}).
		\ea
		\eeq
		It is remarkable that the reasoning works for \textit{any choice} of the convex distance function $d$.
		Taking into account that $d$ is strongly convex~\eqref{Bregman} with respect to a fixed primal norm, and using the bound for the initial distance, we obtain:
		\beq \label{BoundGenMin2}
		\ba{rcl}
		\| \vv_k^{\star} - \xx^{\star} \|^2 + \| \vv_k^{\star} - \vv_k \|^2 & \leq & R^2.
		\ea
		\eeq
		
		We can continue the reasoning as follows, using the convexity of $f$ and $\psi$:
		$$
		\ba{rcl}
		& & \!\!\!\!\!\!\!\!\!\!\!\!\!\!\!\!\!\!\!\!
		\beta_d(\xx_0; \xx^{\star}) + A_{k + 1} F(\xx^{\star})  \\
		\\
		&\overset{(\ref{GenProof1})}{\geq} &
		\beta_d(\vv_k; \xx^{\star}) + a_{k + 1} F(\xx^{\star}) + A_{k} F(\xx_k)
		\\
		\\
		& \geq & 
		\beta_d(\vv_k; \xx^{\star})
		+ a_{k + 1} \psi(\xx^{*}) + A_{k} \psi(\xx_k)
		+ A_{k + 1} f(\xx_{k + 1}) + a_{k + 1} \la \nabla f(\xx_{k + 1}), \xx^{\star} - \vv_{k + 1} \ra \\
		\\
		& \geq &
		\beta_d(\vv_k; \xx^{\star})
		+ a_{k + 1} \psi(\vv_{k + 1}) + A_{k} \psi(\xx_k)
		+ A_{k + 1} f(\xx_{k + 1}) \\
		\\
		& & \qquad + \; a_{k + 1} \la \nabla f(\xx_{k + 1}) + \psi'(\vv_{k + 1}), \xx^{\star} - \vv_{k + 1} \ra  \\
		\\
		& \geq & 
		\beta_d(\vv_k; \xx^{\star}) + A_{k + 1} F(\xx_{k + 1})
		+ a_{k + 1} \la \nabla f(\xx_{k + 1}) + \psi'(\vv_{k + 1}), \xx^{\star} - \vv_{k + 1} \ra.
		\ea
		$$
		Now, we use the key identity~\eqref{BregmanMain}, which explains the structure of the Bregman divergence:
		$$
		\ba{rcl}
		\beta_d(\vv_k; \xx^{\star}) & = & \beta_{d}(\vv_{k + 1}; \xx^{\star})
		+ \beta_d(\vv_{k}; \vv_{k + 1})
		+ \la \nabla d(\vv_k) - \nabla d(\vv_{k + 1}), \vv_{k + 1} - \xx^{\star} \ra.
		\ea
		$$
		Hence,
		$$
		\ba{rcl}
		& & \!\!\!\!\!\!\!\!\!\!\!\!\!\!\!\!\!\!\!\!
		\beta_d(\xx_0; \xx^{\star}) + A_{k + 1} F(\xx^{\star}) \\
		\\
		& \geq & \beta_d(\vv_{k + 1}; \xx^{\star})
		+ A_{k + 1} F(\xx_{k + 1})
		+ \beta_d(\vv_k; \vv_{k + 1}) \\
		\\
		& & \qquad + \; a_{k + 1}\la \nabla f(\xx_{k + 1}) + \psi'(\vv_{k + 1})
		+ \frac{1}{a_{k + 1}}( \nabla d(\vv_{k + 1}) - \nabla d(\vv_{k}) ), \xx^{\star} - \vv_{k + 1}  \ra \\
		\\
		& \overset{(\ref{DeltaKDef})}{\geq}  &
		\beta_d(\vv_{k + 1}; \xx^{\star})
		+ A_{k + 1} F(\xx_{k + 1})
		+ \beta_d(\vv_k; \vv_{k + 1}) 
		- \frac{\ell_k}{2}\| \vv_{k + 1} - \vv_k \|^2 \cdot  \| \vv_{k + 1} - \xx^{\star} \|
		\\
		\\
		& & \qquad + \; a_{k + 1}\la \nabla f(\xx_{k + 1}) 
		- \nabla f(\yy_k) - \mat{H}_k(\xx_{k + 1} - \yy_k), \xx^{\star} - \vv_{k + 1}  \ra \\
		\\
		& \overset{(\ref{GradApprox})}{\geq} &
		\beta_d(\vv_{k + 1}; \xx^{\star})
		+ A_{k + 1} F(\xx_{k + 1})
		+ \beta_d(\vv_k; \vv_{k + 1}) 
		- 
		\ell_k \| \vv_{k + 1} - \vv_k \|^2
		\cdot \| \vv_{k + 1} - \xx^{\star} \| \\
		\\
		& \overset{(\ref{Bregman})}{\geq}&
		\beta_d(\vv_{k + 1}; \xx^{\star})
		+ A_{k + 1} F(\xx_{k + 1})
		+ \frac{1}{2} \| \vv_k - \vv_{k + 1} \|^2 \Bigl( 
		1 - 2 \ell_k \| \vv_{k + 1} - \xx^{\star} \|
		\Bigr)
		\ea
		$$
		where $\boxed{\ell_k = \tfrac{a_{k + 1}^3 L_2}{A_{k + 1}^2}}$.
		Therefore, to complete the proof we need to check the following inequality:
		\beq \label{InequalityToCheck}
		\ba{rcl}
		\ell_k \| \vv_{k + 1} - \xx^{\star} \| & \leq & \frac{1}{2}.
		\ea
		\eeq
		Let us observe the main properties of the inexact composite Newton step 
		equipped with the condition~\eqref{DeltaKDef}.
		Due to strong convexity of the distance function, and using convexity of the objective, so that $\mat{H}_k \succeq \0$, we have
		$$
		\ba{rcl}
		& & \!\!\!\!\!\!\!\!\!\!\!\!\!\!\!\!\!\!\!\!\!\!\!\!\!\!\!\!\!\!
		 \| \vv_{k + 1} - \vv_k^{\star} \|^2 \\
		\\
		& \leq & \la \nabla d(\vv_{k + 1}) - \nabla d(\vv_k^{\star}), \vv_{k + 1} - \vv_k^{\star} \ra \\
		\\
		& \leq & \la \nabla d(\vv_{k + 1}) - \nabla d(\vv_k^{\star}) + \gamma_k a_{k + 1} \mat{H}_k(\vv_{k + 1} - \vv_k^{\star}), \vv_{k + 1} - \vv_k^{\star} \ra \\
		\\
		& \overset{(\ref{BregmanOptimality})}{=} &
		\la \nabla d(\vv_{k + 1}) 
		- \nabla d(\vv_k) + a_{k + 1} \nabla f(\yy_k) + \gamma_k a_{k + 1} \mat{H}_k (\vv_{k + 1} - \vv_k) \\
		\\
		& & \qquad
		+ \;a_{k + 1} \psi'(\vv_{k + 1}),
		\vv_{k + 1} - \vv_k^{\star} \ra \\
		\\
		& & + \;
		a_{k + 1} \la \nabla f(\xx_{k + 1}^{\star}) - \nabla f(\yy_k) - \gamma_k \mat{H}_k (\vv_k^{\star} - \vv_k),
		\vv_{k + 1} - \vv_k^{\star} \ra \\
		\\
		& & + \; \underbrace{a_{k + 1} \la \psi'(\vv_k^{\star}) - \psi'(\vv_{k + 1}), \vv_{k + 1} - \vv_k^{\star} \ra}_{\leq 0} \\
		\\
		& \overset{(\ref{DeltaKDef}), (\ref{GradApprox})}{\leq} &
		\bigl( a_{k + 1}\delta_{k + 1} +  \frac{\ell_k}{2} \| \vv_{k} - \vv_k^{\star} \|^2 \bigr) \| \vv_{k + 1} - \vv_k^{\star} \|.
		\ea
		$$
		Therefore, we justified that
		\beq \label{VkNewBregmanBound}
		\ba{rcl}
		\| \vv_{k + 1} - \vv_k^{\star} \| & \leq & 
		a_{k + 1}\delta_{k + 1} + \frac{\ell_k}{2} \| \vv_k - \vv_k^{\star}\|^2
		\;\; \overset{(\ref{DeltaKDef})}{\leq} \;\;
		\frac{\ell_k}{2} \| \vv_k - \vv_k^{\star} \|^2 + \frac{\ell_k}{2} R^2.
		\ea
		\eeq
		It remains to notice that
		$$
		\ba{rcl}
		\ell_k \| \vv_{k + 1} - \xx^{\star} \|
		& \leq & 
		\ell_k \| \vv_k^{\star} - \xx^{\star} \|
		+ \ell_k \| \vv_{k + 1} - \vv_k^{\star} \| \\
		\\
		& \overset{(\ref{VkNewBregmanBound})}{\leq} &
		\ell_k \| \vv_k^{\star} - \xx^{\star} \|
		+ \frac{\ell_k^2}{2} \| \vv_k - \vv_k^{\star}\|^2 + \frac{\ell_k^2}{2} R^2 \\
		\\
		& \overset{(\ref{BoundGenMin2})}{\leq} &
		\max\limits_{\substack{(a, b) \in \R^2 \\ a^2 + b^2 \leq R^2 }}
		\Bigl[  \ell_k a + \frac{\ell_k^2}{2} b^2 \Bigr]
		+ \frac{\ell_k^2}{2} R^2
		\;\; \leq \;\;
		\ell_k R + \frac{\ell_k^2}{2} R^2 \;\; \leq \;\; \frac{1}{2},
		\ea
		$$
		where we used our selection of the parameters~\eqref{NewAkChoice}.
	\end{proof}
	
	\begin{corollary}
		We see that up to a numerical constant in the right hand side, condition~\eqref{NewAkChoice} 
		coincides with that one from Theorem~\ref{TheoremMain}.
		For example, we can set
		$$
		\ba{rcl}
		A_{k + 1} & = & \frac{2}{3^3 5 L_2 R}(k + 1)(k + 2)(k + 3),
		\qquad
		a_{k + 1} \;\; = \;\; \frac{2}{3^2 5 L_2 R}(k + 1) (k + 2)
		\ea
		$$
		and $\gamma_k = \frac{a_{k + 1}}{A_{k + 1}} = \frac{3}{k + 3}$
		in our general algorithm, which gives the following accelerated rate, for $k \geq 1$:
		$$
		\ba{rcl}
		F(\xx_k) - F^{\star} & \overset{(\ref{NewInvariant})}{\leq} & \frac{\beta_d(\xx_0; \xx^{\star})}{A_k}
		\;\; \leq \;\;
		\frac{3^3 5 L_2 R^3}{k(k + 1)(k + 2)}
		\;\; = \;\;
		O\Bigl( \frac{L_2 R^3}{k^3}  \Bigr).
		\ea
		$$
	\end{corollary}

	\section{Discussion}
	
	In this work, we show that the direct inclusion of second-order information 
	into the classic fast gradient method in the momentum form~\eqref{FNM} 
	achieves an improved convergence rate of $O(1 / k^3)$
	for convex functions with Lipschitz continuous Hessian.	This is the same rate as for the 
	accelerated cubic Newton from~\cite{nesterov2008accelerating}. However, in our approach we
	use only one (possibly inexact) linear solve per iteration, and do not require performing correcting extragradient steps. To the best of our knowledge, this is the first second-order method
	with such simple iteration complexity for this problem class.
	
	One of the main open questions is how to perform an adaptive estimation of the parameters from~\eqref{PrimalNewton} such as $L_2$ and $R$. While currently we can treat their product $\gamma \equiv L_2 R$ as a single scalar second-order step-size that we fine-tune, it would be interesting
	to develop an adaptive second-order scheme with only one linear solve per iteration.
	
	It would be also interesting to explore optimal second-order schemes with an improved $O(1 / k^{7/2})$ rate and simple complexity per iteration. This seems to require replacing our global parameter $R$ by a local estimate of the proximal step, which immediately leads to complications of the subproblem, which is no longer linear.
	
	Finally, it would be interesting to establish connections to classic quasi-Newton methods, 
	by combining our analysis with recent developments of their non-asymptotic analysis~\cite{rodomanov2021new,rodomanov2022quasi,jin2023non}.
	We keep these directions for future research.
	
	\section*{Acknowledgment}
	
	We are very grateful to Yurii Nesterov and Anton Rodomanov for useful comments on a preliminary version of this manuscript, which improved the presentation.

\end{document}